\documentclass[10pt]{amsart}
\usepackage{fullpage}
\usepackage{a4wide,amsmath,amssymb,amsfonts,color,eucal,graphicx,latexsym,verbatim,url}

\newtheorem{theorem}{Theorem}

\newtheorem{algorithm}[theorem]{Algorithm}

\newtheorem{corollary}[theorem]{Corollary}

\newtheorem{lemma}[theorem]{Lemma}

\newtheorem{remark}[theorem]{Remark}

\numberwithin{theorem}{section}
\numberwithin{equation}{section}

\let\Im\relax
\DeclareMathOperator{\Im}{Im}
\let\Re\relax
\DeclareMathOperator{\Re}{Re}
\newcommand{\C}{\mathbb{C}}
\renewcommand{\H}{\mathbb{H}}
\newcommand{\N}{\mathbb{N}}
\newcommand{\R}{\mathbb{R}}
\newcommand{\Z}{\mathbb{Z}}
\newcommand{\eps}{\varepsilon}
\DeclareMathOperator{\erfc}{erfc}
\DeclareMathOperator{\sech}{sech}
\DeclareMathOperator{\sinc}{sinc}
\DeclareMathOperator{\SL}{SL}

\newcommand{\F}[4]{{{}_2F_1\left({#1},\,{#2};\,{#3};\,{#4}\right)}}

\newcommand{\cc}[1]{\overline{#1}}
\newcommand{\abs}[1]{\left\vert#1\right\vert}

\begin{document}
\title{Rapid computation of $L$-functions attached to Maass forms}
\author[A.~R.\ Booker]{Andrew~R.\ Booker}
\address{Department of Mathematics, University of Bristol, University Walk,
  Bristol, BS8 1TW, United Kingdom, e-mail: andrew.booker@bristol.ac.uk}
\author[H.~Then]{Holger Then}
\address{Alemannenweg 1, 89537 Giengen, Germany,
  e-mail: holger.then@bristol.ac.uk}

\begin{abstract}
Let $L$ be a degree-$2$ $L$-function associated to a Maass cusp form.
We explore an algorithm that evaluates $t$ values of
$L$ on the critical line in time $O(t^{1+\eps})$.
We use this algorithm to rigorously compute
an abundance of consecutive zeros and investigate their distribution.
\end{abstract}

\thanks{The authors wish to express their thanks to Andreas Str\"ombergsson
and Pankaj Vishe for offering deep insight into their methods.  H.~T.\ thanks
Brian Conrey, Dennis Hejhal, Jon Keating, Anton Mellit, and Franzesco Mezzadri
for inspiring discussions. A.~B.\ and H.~T.\ acknowledge support from EPSRC
grant EP/H005188/1.}
\date{\today}

\maketitle

\section{Introduction}\label{s:intro}

In \cite{Booker2006}, the first author presented an algorithm for the
rigorous computation of $L$-functions associated to automorphic forms.
The algorithm is efficient when one desires many values of a
single $L$-function or values of many $L$-functions with a common
$\Gamma$-factor. In this paper, we explore the prototypical case of a
family of $L$-functions to which that does not apply, namely Maass cusp
forms in the eigenvalue aspect.

As described in \cite[\S5]{Booker2006}, one of the main challenges when
computing $L$-functions is the evaluation of the inverse Mellin transform
of the associated $\Gamma$-factor. Rubinstein \cite{Rubinstein2005} describes
an algorithm based on continued fractions that performs well in practice,
but for which it seems to be very difficult to obtain rigorous error
bounds. On the other hand, the algorithm in \cite{Booker2006}, following
Dokchitser \cite{Dokchitser2004}, uses a precomputation based on simpler
power series expansions that are easy to make rigorous;
it works well for motivic $L$-functions of low
weight, but suffers from catastrophic precision loss when the shifts in
the $\Gamma$-factor grow large, as is the case for Maass forms.

A well-known similar problem occurs when one attempts to evaluate an
$L$-function high up in the critical strip. Rubinstein, following an idea of
Lagarias and Odlyzko \cite{LagariasOdlyzko1979}, has demonstrated that this can
be dealt with effectively by multiplying by an exponential factor to compensate
for the decay of the $\Gamma$-factor; specifically, for a complete
$L$-function $\Lambda(s)$ of degree $d$, one works with $e^{-i\theta
s}\Lambda(s)$ for a suitable $\theta<\pi d/4$.  This idea can be made
to work for general $L$-functions, including those associated to Maass
forms (albeit with the problems related to precision loss noted above,
if the $\Gamma$-factor is not fixed), and Molin \cite{Molin2010} has worked
out rigorous numerical methods in quite wide generality.

For the specific case of Maass cusp forms, Vishe \cite{Vishe2012} (see also
\cite{Good1983}) has shown that the ``right''
factor to multiply by to account for the variation in both $t$ and the
$\Gamma$-factor shifts is not the exponential function $e^{-i\theta s}$,
but rather the hypergeometric function
\begin{equation}\label{e:hypgeom}
\F{\frac{s+\epsilon+ir}2}{\frac{s+\epsilon-ir}2}{\frac{1}{2}+\epsilon}{
  -\tan^2\theta}
\end{equation}
where $\epsilon$ denotes the parity of the Maass form, and
$\frac{1}{4}+r^2$ is its Laplacian eigenvalue.
To understand the motivation for this factor, consider first the case
of a classical holomorphic cuspform $f$, for which
the $L$-function is defined via the Mellin transform
$$
\Lambda(s)=\int_0^\infty f(iy)y^s\frac{dy}y.
$$
Since $f$ is holomorphic and vanishes in the cusp, we can change the
contour of integration
from $(0,\infty)$ to $e^{i\theta}(0,\infty)$ for some
$\theta\in(-\pi/2,\pi/2)$; writing $y=e^{i\theta}u$ for $u\in\R$, we
obtain
$$
e^{-i\theta s}\Lambda(s)=\int_0^\infty f(ie^{i\theta}u)u^s\frac{du}u.
$$
Thus, Rubinstein's exponential factor arises naturally from a contour
rotation.

For a Maass form $f$ of weight $0$ and even parity (say),
we similarly have the integral representation
$$
\Lambda(s)=\int_0^\infty f(iy)y^{s-\frac{1}{2}}\frac{dy}{y}.
$$
Since $f$ is no longer holomorphic in this case, we cannot simply rotate the
contour in this expression, but we are free to start with the rotated
integral $\int_0^\infty f(ie^{i\theta}u)u^{s-\frac{1}{2}}\frac{du}u$ and try
to relate it back to $\Lambda(s)$. As we show in \S\ref{s:prelim},
this can be done, and the two are related essentially by the factor
\eqref{e:hypgeom}.

We analyze this strategy in greater detail in \S\ref{s:prelim}, but the
upshot is that to compute Maass form $L$-functions for a wide range of
values of $r$ and $t$, it suffices to compute $f(ie^{i\theta}u)$ for
suitable values of $\theta$ and $u$.  In turn, using modularity to move
each point to the fundamental domain, the problem reduces to computing
the $K$-Bessel function $K_{ir}(y)$ for various $r$ and $y$. Fortunately,
that is a problem that underlies all computational aspects of Maass cusp
forms and has been well studied; see, for instance,
\cite{BookerStrombergssonThen2013}.

\subsection*{Numerical results}

In \S\ref{s:results}, using as input the rigorous numerical Maass form data of
\cite{Bian2017} we compute values of the corresponding Maass form
$L$-functions and use the resulting numerical data to test conjectures about
the distribution of zeros of Maass form $L$-functions in the $t$- and
$r$-aspects.  In particular, we show that the phenomenon of zero repulsion
around $\frac{1}{2}\pm ir$ that Str\"ombergsson observed
\cite{Strombergsson1999} disappears in the large eigenvalue limit.

We derive rigorous error estimates and use the interval arithmetic
package MPFI \cite{mpfi} throughout our computations to manage round-off errors.
Thus, modulo bugs in the software or hardware, our numerical results are
rigorous.

\section{Preliminaries on Maass forms}\label{s:prelim}

Let $f\in L^2(\Gamma_1(N)\backslash\H)$ be a cuspidal Maass newform
and Hecke eigenform of weight $0$ and level $N$. Then $f$ has a Fourier
expansion of the form
$$
f(x+iy)=\sqrt{y}\sum_{n=1}^\infty a_nK_{ir}(2\pi ny)\cos^{(-\epsilon)}(2\pi nx),
$$
where $a_n$ is the $n$th Hecke eigenvalue of $f$, $\frac{1}{4}+r^2$ is its
eigenvalue for the hyperbolic Laplacian
$-y^2\left(\frac{\partial^2}{\partial x^2}
+\frac{\partial^2}{\partial y^2}\right)$,
$\epsilon\in\{0,1\}$ indicates the parity,
and
$$
\cos^{(-\epsilon)}(x):=
\begin{cases}
\cos{x}&\text{if }\epsilon=0,\\
\sin{x}&\text{if }\epsilon=1.
\end{cases}
$$
Moreover, $f$ is related to its dual $\bar{f}$ via the Fricke
involution, so that
\begin{equation}\label{e:fricke}
f(z)=w\overline{f\!\left(-\frac{1}{Nz}\right)},
\end{equation}
for some complex number $w$ with $|w|=1$.

Associated to $f$ we have the $L$-function $L(s,f)$, defined
for $\Re(s)>1$ by the series
$$
L(s,f):=\sum_{n=1}^\infty\frac{a_n}{n^s}.
$$
It follows from \eqref{e:fricke} that $L(s,f)$ continues to an entire
function and satisfies a functional equation relating its values
at $s$ and $1-\bar{s}$. To see this,
let ${}_2F_1$ denote the Gauss hypergeometric function
$$\F{\alpha}{\beta}{\gamma}{z}:=
1+\sum_{n=1}^{\infty} \frac{(\alpha)_n(\beta)_n z^n}{(\gamma)_n n!},$$
and consider the family of $\Gamma$-factors
\begin{multline*}
  \gamma_\theta(s,f):=
  i^{-\epsilon}w^{-1/2}\left(\frac{\cos\theta}{\sqrt{N}}\right)^{1/2-s}
  \Gamma_\R(s+\epsilon+ir)\Gamma_\R(s+\epsilon-ir)
  \\
  \cdot\F{\frac{s+\epsilon+ir}{2}}{\frac{s+\epsilon-ir}{2}}{
    \frac{1}{2}+\epsilon}{-\tan^2\theta},
\end{multline*}
where $\Gamma_\R(s):=\pi^{-s/2}\Gamma(s/2)$ and
$\theta\in(-\pi/2,\pi/2)$ is a parameter.
By \cite[Sec.~6.699, Eqs.~3 and 4]{GradshteynRyzhik2007},
we have
\begin{align}
  \label{gamma}
  \gamma_\theta(s,f) = 4 i^{-\epsilon} w^{-1/2}
  \left(\frac{\cos\theta}{\sqrt{N}}\right)^{1/2-s}
  \frac{1}{\pi^{s+\epsilon}}
  \int_0^\infty K_{ir}(2t)
  \frac{\cos^{(-\epsilon)}(2\tan(\theta)t)}{(2\tan(\theta)t)^\epsilon}
  t^{s+\epsilon} \frac{dt}{t}
\end{align}
for $\Re s>0$.
(Note that for a Maass form with odd reflection symmetry, viz.\ $\epsilon=1$,
\eqref{gamma} has a removable singularity at $\theta=0$; this is
related to the fact that the complete $L$-function is the Mellin
transform of $\partial{f}/\partial{x}$ rather than $f$.)
Making the substitution $t\mapsto\pi n\cos(\theta)u$, we can express
the complete $L$-function
$\Lambda_\theta(s,f):=\gamma_\theta(s,f)L(s,f)$
as the Mellin transform of the Maass form along a ray in the upper half
plane:
\begin{align}
  \Lambda_\theta(s,f)
  & = \gamma_\theta(s,f) L(s,f)
  \nonumber \\
  & = c_\theta(s,f)
  \sum_{n=1}^\infty a_n \int_0^\infty (\cos(\theta)u)^{1/2}
  K_{ir}(2\pi n\cos(\theta)u)
  \cos^{(-\epsilon)}(2\pi n\sin(\theta)u) u^{s-1/2} \frac{du}{u}
  \nonumber \\
  \label{Lambda} & = c_\theta(s,f)
  \int_0^\infty f(ie^{i\theta}u) u^{s-1/2} \frac{du}{u},
  \qquad \text{where }
  c_\theta(s,f) := \frac{4 w^{-1/2}}{(2\pi i\tan\theta)^\epsilon}
  N^{\frac{1}{2}(s-\frac{1}{2})}.
\end{align}
Splitting the integral at $u=1/\sqrt{N}$ and employing
\eqref{e:fricke} completes the analytic continuation:
$$\Lambda_\theta(s,f)
= c_\theta(s,f)
\left\{ \int_{1/\sqrt{N}}^\infty f(ie^{i\theta}u) u^{s-1/2} \frac{du}{u}
+ \int_0^{1/\sqrt{N}} w\cc{f\!\left(-\frac{1}{Nie^{i\theta}u}\right)} u^{s-1/2}
\frac{du}{u}\right\}.$$
Using that $f(-\cc{z}) = (-1)^\epsilon f(z)$ and making the
substitution
$u\mapsto\frac{1}{Nu}$, we obtain the functional equation:
\begin{align*}
\Lambda_\theta(s,f)
&=c_\theta(s,f)
\int_{1/\sqrt{N}}^\infty f(ie^{i\theta}u) u^{s-1/2} \frac{du}{u}
+ \cc{c_\theta(1-\cc{s},f)}
\int_{1/\sqrt{N}}^\infty \cc{f(ie^{i\theta}u)} u^{1/2-s} \frac{du}{u}\\
&= \cc{\Lambda_\theta(1-\cc{s},f)}.
\end{align*}
In particular, $\Lambda_\theta(1/2+it,f)\in\R$ for $t\in\R$.

\section{Rigorous computation of $L$-functions}\label{s:alg}

We describe an algorithm based on the fast Fourier transform
that allows one to evaluate $\Lambda_\theta(s,f)$ quickly, if one is
interested in many points.

The integral~\eqref{Lambda} is essentially a Fourier transformation,
\begin{align}
  \tag{\ref{Lambda}a} & \Lambda_\theta(\sigma+it,f) = c_\theta(\sigma+it,f)
  \int_{-\infty}^{\infty} f(ie^{i\theta}e^u) e^{u(\sigma-1/2)} e^{iut}\,du.
\end{align}
Similarly for the integral~\eqref{gamma},
\begin{multline}\label{gamma a}
  \tag{\ref{gamma}a} \gamma_\theta(\sigma+it,f) \\
  = c_\theta(\sigma+it,f) \int_{-\infty}^{\infty} (\cos(\theta)e^u)^{1/2}
  K_{ir}(2\pi\cos(\theta)e^u)
  \cos^{(-\epsilon)}(2\pi\sin(\theta)e^u) e^{u(\sigma-1/2)} e^{iut}\,du.
\end{multline}

In order to use the fast Fourier transform to compute
$$g(t) = \int_{-\infty}^{\infty} \hat{g}(u) e^{iut}\,du,$$
we first need to discretize the integral.  To that end, let $A,B>0$ be
parameters such that $q=AB$ is an integer.  In the Poisson summation
formula,
$$\sum_{k\in\Z} g\!\left(\frac{m}{A}+kB\right) = \frac{2\pi}{B} \sum_{l\in\Z}
e\!\left(\frac{ml}{q}\right) \hat{g}\!\left(\frac{2\pi l}{B}\right),$$
we solve for $g(\frac{m}{A})$, which results in
$$g\!\left(\frac{m}{A}\right) = \frac{2\pi}{B} \sum_{l=-C'}^{C}
e\!\left(\frac{ml}{q}\right) \hat{g}\!\left(\frac{2\pi l}{B}\right) + \eps_g,$$
$$\eps_g :=
\frac{2\pi}{B} \sum_{\{l\in\Z:l<-C'\text{ or }l>C\}}
e\!\left(\frac{ml}{q}\right) \hat{g}\!\left(\frac{2\pi l}{B}\right)
- \sum_{k\neq0} g\!\left(\frac{m}{A}+kB\right).$$
In \S\ref{s:bounds} we will derive precise bounds for
this error term.

\section{Bounds}\label{s:bounds}

Let $Q(s,f)$ be the analytic conductor, defined by
$$Q(s,f) := N \frac{s+\epsilon+ir}{2\pi} \frac{s+\epsilon-ir}{2\pi}.$$
Note that $\gamma(s,f)$ satisfies the recurrence
$\gamma(s+2,f) = Q(s,f) \gamma(s,f)$.
Further, we define
$$\chi(s,f) := \frac{\cc{\gamma(1-\cc{s},f)}}{\gamma(s,f)},$$
so that $L(s,f) = \chi(s,f) \cc{L(1-\cc{s},f)}$.

\begin{lemma}\cite[\S4]{Booker2006}\label{lem:Booker4.1}
For $s$ in the strip $\{s\in\C : -\frac{1}{2}\leq\Re s\leq\frac{3}{2}\}$,
$$\abs{L(s,f)}^2 \leq \abs{\chi(s,f)Q(s,f)}
\sup_{t\in\R}\abs{L(\tfrac{3}{2}+it,f)}^2.$$
\end{lemma}

\begin{remark}\rm
The estimate in Lemma~\ref{lem:Booker4.1} is not optimal since, for
$s=\frac{1}{2}+it$ for large $t$, the right-hand side grows quadratically in
$t$, whereas the convexity estimate would give $O(t^{1+\varepsilon})$.
Moreover, for $\epsilon=0$ and
$s=1\pm ir$ the bound becomes useless, since $\abs{L(1\pm ir,f)}<\infty$,
whereas $\lim_{s\to1\pm ir}\abs{\chi(s,f)Q(s,f)}\to\infty$. Nevertheless, the
estimate is clean and uniform in all parameters, and suffices for our
purposes if we keep away from $s=1\pm ir$.
\end{remark}

\begin{corollary}\label{cor:L}
For $s$ in the strip $\{s\in\C : \frac{1}{2}\leq\Re s<1\}$,
$$\abs{L(s,f)} \leq 3N^{1/2}(\abs{\Im s}+D_{s,f})$$
with
$$D_{s,f} := 3\Re s -1 + \epsilon + \abs{r}
+ \frac{(2\Re s-1)^2}{1-\Re s+\epsilon}.$$
\end{corollary}

\begin{proof}
Recall that $\Gamma_\R(s)$ satisfies the recurrence, reflection, and duplication
formulas
$$s \Gamma_\R(s) = 2\pi \Gamma_\R(2+s),
\quad \Gamma_\R(s)\Gamma_\R(2-s) = \frac{1}{\sin(\frac{\pi}{2}s)},
\quad \text{and} \quad \Gamma_\R(s)\Gamma_\R(1+s) = 2^{1-s} \Gamma_\R(2s).$$
Hence, for $\epsilon\in\{0,1\}$ and $t\in\R$,
\begin{align*}
  \abs{\frac{\Gamma_\R(2+\epsilon-it)}{\Gamma_\R(1+\epsilon+it)}}^2
  = \left.\begin{cases}
    \frac{t^2 \cosh(\frac{\pi}{2}t)}{2\pi t \sinh(\frac{\pi}{2}t)}
    & \text{for }\epsilon=0
    \\
    \frac{(1+t^2) \sinh(\frac{\pi}{2}t)}{2\pi t \cosh(\frac{\pi}{2}t)}
    & \text{for }\epsilon=1
  \end{cases}\right\}
  \leq \frac{(1+\epsilon)^2+t^2}{\pi^2},
\end{align*}
which yields
$$\abs{\chi(s,f)\cc{Q(1-\cc{s},f)}} \leq 4N^{-1/2}\abs{Q(s,f)}$$
for $\Re s=1$.

On the critical line we have
$$\abs{\chi(s,f)\frac{\cc{Q(1-\cc{s},f)}}{Q(s,f)}}_{\Re s=\frac{1}{2}} = 1,$$
and by the Phragm\'en-Lindel\"of theorem,
$$\sup_{\frac{1}{2}\leq\Re s\leq1}\abs{\chi(s,f)\frac{\cc{Q(1-\cc{s},f)}}{Q(s,f)}}
\leq \max\{4N^{-1/2},1\} \leq 4.$$
Thus, in the strip $\frac{1}{2}\leq\Re s<1$ we have
$$\abs{\frac{Q(s,f)^2}{\,\cc{Q(1-\cc{s},f)}\,}} \leq
N\left( \frac{\abs{\Im s}+D_{s,f}}{2\pi} \right)^2.$$
Using the Kim--Sarnak estimate \cite{KimSarnak2003}
$p^{-\vartheta}\leq\abs{\alpha_p}\leq p^{\vartheta}$ with
$\vartheta=\frac{7}{64}$ in the Euler product gives
$$
\sup_{\Re s=\frac{3}{2}}\abs{L(s,f)} = \sup_{\Re s=\frac{3}{2}}
\abs{\prod_p\frac{1}{(1-\alpha_pp^{-s})(1-\alpha_p^{-1}p^{-s})}}
\leq \zeta(\tfrac{3}{2}+\vartheta)\zeta(\tfrac{3}{2}-\vartheta) < 3\pi.
$$
Inserting the last three bounds in Lemma~\ref{lem:Booker4.1} yields the
corollary.
\end{proof}

\begin{lemma}\label{lem:gamma}
For $s=\sigma+it$ with $0<\sigma\leq1$ and $0<\theta<\delta<\frac{\pi}{2}$,
$$\abs{\gamma_\theta(\sigma+it,f)} <
E_{\sigma,\theta,\delta} e^{-(\delta-\theta)\abs{t}}$$
with
\begin{align*}
E_{\sigma,\theta,\delta}
:= \frac{\abs{c_\theta(\sigma,f)}}{(\cos(\delta-\theta))^{1/2}}
\left\{
\frac{\cosh(1)\sigma^{-1}(\sigma^{-1}+\log(2)+e^{-1})}{(2\pi)^{\sigma}
  (\cos(\delta-\theta)\cos\theta)^{\sigma-1/2}}
+ \frac{\Gamma(\sigma-\frac{1}{2},
  \frac{\cos\delta}{\cos(\delta-\theta)\cos\theta})}{
  2(2\pi\cos\delta)^{\sigma-1/2}} \right\}.
\end{align*}
\end{lemma}

\begin{proof}
For $\gamma_\theta$ we have the integral representation~\eqref{gamma a}.
Since $\abs{\gamma_\theta(\sigma-it,f)} = \abs{\gamma_\theta(\sigma+it,f)}$,
it is enough to prove the lemma for non-negative $t$.

Making the change of variables $u\mapsto u+i(\delta-\theta)$
and moving the contour
of integration back to the real line, we get
\begin{align*}
  g(t) := \gamma_\theta(\sigma+it,f) =
  \int_{\R} \hat{g}(u) e^{iut}\,du =
  \int_{\R} \hat{g}(u+i(\delta-\theta)) e^{i(u+i(\delta-\theta))t}\,du,
\end{align*}
and
\begin{align*}
  \abs{g(t)} \leq e^{-(\delta-\theta)t} \left\{
  \int_{\{u\in\R:2\pi e^u\cos(\delta-\theta)\cos\theta < 1\} \hspace*{-9em}}
  \abs{\hat{g}(u+i(\delta-\theta))}\,du \ \quad
  +
  \int_{\{u\in\R:2\pi e^u\cos(\delta-\theta)\cos\theta \geq 1\} \hspace*{-9em}}
  \abs{\hat{g}(u+i(\delta-\theta))}\,du \ \quad
\right\}.
\end{align*}
We bound the first integral using the estimates
$\abs{K_{ir}(z)} \leq \log\left(\frac{2}{\Re z}\right)+e^{-1}$
\cite[p.~106]{BookerStrombergssonThen2013} and
$\abs{\cos^{(-\epsilon)}(\tan(\delta-\theta)\tan(\theta)\Im z)}\leq\cosh(1)$,
and the second integral using
$\abs{K_{ir}(z)}<\left(\frac{\pi}{2\Re z}\right)^{1/2} e^{-\Re z}$ and
$\abs{\cos^{(-\epsilon)}(z)} \leq e^{\abs{\Im z}}$.
\end{proof}

\begin{lemma}\label{lem:sum gamma}
For $s=\sigma+it$ with $0<\sigma\leq1$,
$\abs{t}=\frac{\abs{m}}{A}\leq T\leq\frac{B}{2}$, and
$0<\theta<\delta<\frac{\pi}{2}$,
\begin{align*}
  \abs{\sum_{k\neq0}
	\gamma_\theta\!\left(\sigma+i\left(\frac{m}{A}+kB\right),f\right)}
  < \frac{E_{\sigma,\theta,\delta}}{\sinh((\delta-\theta)\frac{B}{2})}.
\end{align*}
\end{lemma}
\begin{proof}
Using Lemma~\ref{lem:gamma} together with
$\abs{\frac{m}{A}+kB}\geq(\abs{k}-\frac{1}{2})B$ yields the stated
bound.
\end{proof}

\begin{lemma}\label{lem:sum Lambda}
For $s=\sigma+it$ with $\frac{1}{2}\leq\sigma<1$,
$\abs{t}=\frac{\abs{m}}{A}\leq T\leq\frac{B}{2}$, and
$0<\theta<\delta<\frac{\pi}{2}$,
\begin{align*}
  \abs{\sum_{k\neq0}
	\Lambda_\theta\!\left(\sigma+i\left(\frac{m}{A}+kB\right),f\right)}
  < \frac{3N^{1/2}E_{\sigma,\theta,\delta}}{\sinh((\delta-\theta)\frac{B}{2})}
  \left( \frac{B}{2}+D_{\sigma,f}+\frac{B}{1-e^{-(\delta-\theta)B}} \right).
\end{align*}
\end{lemma}
\begin{proof}
By Corollary~\ref{cor:L} and Lemma~\ref{lem:gamma},
\begin{align}\label{Lambda bound}
  \abs{\Lambda_\theta(\sigma+it,f)} < 3N^{1/2}(\abs{t}+D_{\sigma,f})
  E_{\sigma,\theta,\delta}e^{-(\delta-\theta)\abs{t}}.
\end{align}
Applying the estimate
$(\abs{k}-\frac{1}{2})B\leq\abs{\frac{m}{A}+kB}\leq(\abs{k}+\frac{1}{2})B$
and summing up results in the stated bound.
\end{proof}

\begin{lemma}\label{lem:^Lambda}
For $\frac{1}{2}\leq\sigma\leq1$, $B>0$,
$C\geq\frac{B}{2\pi}\log\frac{1+B/(2\pi)}{4\pi\cos\theta}$,
$C'=C+\frac{B}{2\pi}\log N$, and
$\hat{\Lambda}_{\sigma,\theta}(u,f):=c_\theta(\sigma,f)f(ie^{i\theta}e^u)
e^{u(\sigma-1/2)}$,
$$\abs{\frac{2\pi}{B} \sum_{\{l\in\Z:l<-C'\text{ or }l>C\}}
e\!\left(\frac{ml}{q}\right)
  \hat{\Lambda}_{\sigma,\theta}\!\left(\frac{2\pi l}{B},f\right)}
< \frac{2\pi}{B} \frac{56 N^{1/4}}{(2\pi\tan\theta)^\epsilon}
e^{\pi\frac{C}{B} - 2\pi\cos(\theta)e^{2\pi\frac{C}{B}}}.$$
\end{lemma}
\begin{proof}
Applying $\abs{a_n}\leq2n^{1/2}$, $\abs{K_{ir}(y)}<(\frac{\pi}{2y})^{1/2}e^{-y}$,
and $\abs{\cos^{(-\epsilon)}(x)}\leq1$ in the Fourier expansion of the Maass
form gives
$$\abs{f(x+iy)}<\frac{1}{e^{2\pi y}-1},$$
and by Fricke involution
\begin{align}\label{f bound}
  \abs{f(ie^{i\theta}e^u)}=\abs{\cc{f(ie^{-i\theta}e^{-u}/N)}}
  <\frac{1}{e^{2\pi\cos(\theta)\max\{e^u,e^{-u}/N\}}-1}.
\end{align}

For $\sigma\leq1$,
$$\abs{\frac{2\pi}{B} \sum_{l>C} e\!\left(\frac{ml}{q}\right)
  \hat{\Lambda}_{\sigma,\theta}\!\left(\frac{2\pi l}{B},f\right)}
< \frac{2\pi}{B} \sum_{l>C} \frac{4N^{1/4}}{(2\pi\tan\theta)^\epsilon}
\frac{e^{\pi\frac{l}{B}}}{e^{2\pi\cos(\theta)e^{2\pi l/B}}-1}.$$
Writing $a:=2\pi\cos\theta$, $u:=2\pi\frac{l}{B}$, $u_0:=2\pi\frac{C}{B}$,
with $a>0$ and $u_0<u$, we have
$$\frac{e^{-a e^u+\frac{u}{2}}}{1-e^{-a e^u}}
< \frac{e^{-a e^{u_0}(1+u-u_0)+\frac{u}{2}}}{1-e^{-a e^{u_0}}}.$$
Summing the resulting geometric series gives
\begin{align*}
  \abs{\frac{2\pi}{B} \sum_{l>C} e\!\left(\frac{ml}{q}\right)
    \hat{\Lambda}_{\sigma,\theta}\!\left(\frac{2\pi l}{B},f\right)}
  & < \frac{2\pi}{B} \frac{4N^{1/4}}{(2\pi\tan\theta)^\epsilon}
  \frac{e^{-a e^{u_0}+\frac{u_0}{2}}}{(1-e^{-a e^{u_0}})
    (1-e^{-(a e^{u_0}-\frac{1}{2})\frac{2\pi}{B}})}
  \\
  & < \frac{2\pi}{B} \frac{4N^{1/4}}{(2\pi\tan\theta)^\epsilon}
  7 e^{-a e^{u_0}+\frac{u_0}{2}}
\end{align*}
for $\frac{2\pi}{B}(a e^{u_0}-\frac{1}{2}) \geq \frac{1}{2}$,
and similarly for the sum over $l<-C$.

For $\frac{1}{2}\leq\sigma\leq1$,
$$\abs{\frac{2\pi}{B} \sum_{l<-C'} e\!\left(\frac{ml}{q}\right)
  \hat{\Lambda}_{\sigma,\theta}\!\left(\frac{2\pi l}{B},f\right)}
< \frac{2\pi}{B} \sum_{l>C'} \frac{4N^{1/4}}{(2\pi\tan\theta)^\epsilon}
\frac{1}{e^{2\pi\cos(\theta)e^{2\pi l/B}/N}-1}.$$
Writing $\frac{a}{N}:=\frac{2\pi\cos\theta}{N}$, $u:=2\pi\frac{\abs{l}}{B}$,
$u_0':=u_0+\log N=2\pi\frac{C'}{B}$, with $a>0$ and $u_0'<u$, we have
$$\frac{e^{-\frac{a}{N}e^u}}{1-e^{-\frac{a}{N}e^u}}
< \frac{e^{-\frac{a}{N}e^{u_0'}(1+u-u_0')}}{1-e^{-\frac{a}{N}e^{u_0'}}},$$
and summing up gives
$$\abs{\frac{2\pi}{B} \sum_{l<-C'} e\!\left(\frac{ml}{q}\right)
  \hat{\Lambda}_{\sigma,\theta}\!\left(\frac{2\pi l}{B},f\right)}
< \frac{2\pi}{B} \frac{4N^{1/4}}{(2\pi\tan\theta)^\epsilon}
\frac{e^{-a e^{u_0}}}{(1-e^{-a e^{u_0}}) (1-e^{-a e^{u_0}\frac{2\pi}{B}})}.
$$
\end{proof}

\begin{lemma}\label{lem:^gamma}
For $\frac{1}{2}\leq\sigma\leq1$, $B>0$,
$C\geq\frac{B}{2\pi}\log\frac{1+B/(2\pi)}{4\pi\cos\theta}$,
$C'\geq\frac{3B}{2\pi}$, and
\\
$\hat{\gamma}_{\sigma,\theta}(u,f) := c_\theta(\sigma,f) (\cos(\theta)e^u)^{1/2}
K_{ir}(2\pi\cos(\theta)e^u) \cos^{(-\epsilon)}(2\pi\sin(\theta)e^u)
e^{u(\sigma-1/2)}$,
\begin{multline*}
  \abs{\frac{2\pi}{B} \sum_{\{l\in\Z:l<-C'\text{ or }l>C\}}
	e\!\left(\frac{ml}{q}\right)
    \hat{\gamma}_{\sigma,\theta}\!\left(\frac{2\pi l}{B},f\right)}
  \\
  < \frac{2\pi}{B} \frac{N^{1/4}}{(2\pi\tan\theta)^\epsilon}
  \left\{
  6 e^{\pi\frac{C}{B}-2\pi\cos(\theta)e^{2\pi\frac{C}{B}}}
  + 23 (\cos\theta)^{1/6} \sech\!\left(\frac{\pi r}{2}\right)
	e^{-\frac{\pi C'}{3B}}
  \right\}.
\end{multline*}
\end{lemma}

\begin{proof}
We have $\abs{K_{ir}(y)}<(\frac{\pi}{2y})^{1/2}e^{-y}$ and
$\abs{\cos^{(-\epsilon)}(x)}\leq1$, so that
$$\abs{\frac{2\pi}{B} \sum_{l>C} e\!\left(\frac{ml}{q}\right)
  \hat{\gamma}_{\sigma,\theta}\!\left(\frac{2\pi l}{B},f\right)}
< \frac{2\pi}{B} \sum_{l>C} \frac{4N^{1/4}}{(2\pi\tan\theta)^\epsilon}
2^{-1} e^{-2\pi\cos(\theta)e^{2\pi\frac{l}{B}}} e^{\pi\frac{l}{B}}.$$
Writing $a:=2\pi\cos\theta$, $u:=2\pi\frac{l}{B}$, $u_0:=2\pi\frac{C}{B}$,
with $a>0$ and $u_0<u$, we have
$$e^{-a e^u + \frac{u}{2}} < e^{-a e^{u_0}(1+u-u_0) + \frac{u}{2}},$$
and summing the geometric series yields
\begin{align*}
  \abs{\frac{2\pi}{B} \sum_{l>C} e\!\left(\frac{ml}{q}\right)
    \hat{\gamma}_{\sigma,\theta}\!\left(\frac{2\pi l}{B},f\right)}
  & < \frac{2\pi}{B} \frac{4N^{1/4}}{(2\pi\tan\theta)^\epsilon} 2^{-1}
  \frac{e^{-a e^{u_0} + \frac{u_0}{2}}}{1 - e^{-(a e^u - \frac{1}{2})\frac{2\pi}{B}}}
  \\
  & < \frac{2\pi}{B} \frac{6N^{1/4}}{(2\pi\tan\theta)^\epsilon}
  e^{-a e^{u_0} + \frac{u_0}{2}}
\end{align*}
for $\frac{2\pi}{B}(a e^{u_0}-\frac{1}{2})\geq\frac{1}{2}$.

The argument is slightly different for the sum over $l<-C$.
Using that
$\abs{\cosh(\frac{\pi r}{2})K_{ir}(y)}<4y^{-1/3}$
\cite[p.~107]{BookerStrombergssonThen2013} and
$\abs{\cos^{(-\epsilon)}(x)}\leq1$, we have
\begin{align*}
  \abs{\frac{2\pi}{B} \sum_{l<-C'} e\!\left(\frac{ml}{q}\right)
    \hat{\gamma}_{\sigma,\theta}\!\left(\frac{2\pi l}{B},f\right)}
  & < \frac{2\pi}{B} \sum_{l>C'} \frac{4N^{1/4}}{(2\pi\tan\theta)^\epsilon}
  \frac{4(2\pi)^{-1/3}}{\cosh(\frac{\pi r}{2})}
  (\cos(\theta)e^{-\frac{2\pi\abs{l}}{B}})^{1/6}
  \\
  & < \frac{2\pi}{B} \frac{23N^{1/4}}{(2\pi\tan\theta)^\epsilon}
  (\cos\theta)^{1/6} \sech\!\left(\frac{\pi r}{2}\right) e^{-\frac{\pi C'}{3B}}
\end{align*}
for  $\frac{\pi C'}{3B}\geq\frac{1}{2}$.
\end{proof}

The following two Lemmata show that we can circumvent division by zero in
$L(s,f)=\frac{\Lambda_\theta(s,f)}{\gamma_\theta(s,f)}$.

\begin{lemma}\label{lem:psi lemma}
For any $y\in\R$ with $\abs{y}>5$, $\abs{\Gamma_\R(x+iy)}$ is an increasing
function of $x\in\R$.
\end{lemma}
\begin{proof}
It suffices to show that $\frac{\partial}{\partial x}\log\abs{\Gamma_\R(x+iy)}
= \Re\psi_\R(x+iy) > 0$.
Suppose first that $x\geq1$.  Then, differentiating Binet's formula for
$\log\Gamma_\R$, we have
\begin{align*}
  \Re\psi_\R(x+iy)
  & = \frac{1}{2} \log(x^2+y^2) - \frac{x}{x^2+y^2} - \frac{1}{2}\log(4\pi)
  - \int_0^\infty \left( 1 - \frac{1}{t} + \frac{2}{e^{2t}-1} \right)
  e^{-xt} \cos(yt) dt
  \\
  & \geq \frac{1}{2}\log(x^2+y^2) - \frac{x}{x^2+y^2} + \psi_\R(1) + 1.
\end{align*}
For $x\geq1$, this is easily seen to be minimum at $x=1$, so we obtain
$$\Re\psi_\R(x+iy)
\geq \frac{1}{2}\log(1+y^2) - \frac{1}{1+y^2} + \psi_\R(1) + 1.$$

For $x\leq1$ we use the reflection formula
$\psi_\R(z)=\psi_\R(2-z)-\pi\cot(\frac{\pi}{2}z)$ to see that
$$\Re\psi_\R(x+iy)
= \Re\psi_\R(2-(x+iy)) - \Re\pi\cot(\tfrac{\pi}{2}(x+iy)),$$
and apply the above to obtain a bound for $\Re\psi_\R(2-(x+iy))$.
We calculate that
$$\Re\cot(\tfrac{\pi}{2}(x+iy))
= \frac{2e^{\pi y}\sin(\pi x)}{1 - 2e^{\pi y}\cos(\pi x) + e^{2\pi y}},$$
and with a little calculus we see that this is bounded in modulus by
$1/\sinh(\pi\abs{y})$.  Thus, altogether we have
$$\Re\psi_\R(x+iy)
\geq \frac{1}{2}\log(1+y^2) - \frac{1}{1+y^2} - \frac{\pi}{\sinh(\pi\abs{y})}
+ \psi_\R(1) + 1$$
for all $x\in\R$ and $y\neq0$.  Note that the right-hand side is strictly
increasing for $y>0$.  Using the known value
$\psi_\R(1)=-\gamma-\frac{1}{2}\log(16\pi)=-2.53587\ldots$,
it is straightforward to see that $\Re\psi_\R(x+iy)$ is positive for
$\abs{y}>5$.
\end{proof}

\begin{lemma}\label{lem:min gamma}
For $r>5$, $\frac{1}{2}\leq\sigma\leq1$,
$0<\theta_1<\theta_2<\frac{\pi}{2}$,
$\cos\theta_1\leq(4+\abs{t^2-r^2})^{-1/2}$, and
$\cos\theta_2=e^{-\frac{\pi}{2r}}\cos\theta_1$,
$$\max\abs{\gamma_{\theta_1}(s,f),\gamma_{\theta_2}(s,f)}
\geq \frac{2}{3} (\cos\theta_2)^{\frac{1}{2}+\epsilon}(2\pi)^{-\epsilon}
\left( \frac{\pi}{r\sinh(\pi r)} \right)^{1/2}.$$
\end{lemma}

\begin{proof}
We follow the proof of \cite{Sarnak1985}, generalizing it and making the implied
constants explicit.  Using \cite[Sec.~9.132, Eq.~1]{GradshteynRyzhik2007} and
writing
$$g_\theta(s,f) := \Gamma_\R(s+\epsilon+ir)
\frac{\Gamma_\R(-2ir)\Gamma_\R(1+2\epsilon)}{\Gamma_\R(1-s+\epsilon-ir)}
\F{\frac{s+\epsilon+ir}{2}}{\frac{1-s+\epsilon+ir}{2}}{ir+1}{\cos^2\theta},$$
we get
\begin{align}\label{eq:gamma_theta}
  \gamma_\theta(s,f)
  = i^{-\epsilon} w^{-1/2} N^{\frac{1}{2}(s-\frac{1}{2})}
  \left\{ (\cos\theta)^{\frac{1}{2}+\epsilon+ir} g_\theta(s,f)
  + (\cos\theta)^{\frac{1}{2}+\epsilon-ir} \cc{g_\theta(\cc{s},f)} \right\}.
\end{align}
By Lemma~\ref{lem:psi lemma}, for $\sigma\geq\frac{1}{2}$ and $\abs{t+r}>5$,
$$\abs{\frac{\Gamma_\R(s+\epsilon+ir)}{\Gamma_\R(1-s+\epsilon-ir)}}
= \abs{\frac{\Gamma_\R(2\sigma-1+1-\sigma+\epsilon+i(t+r))}{
    \Gamma_\R(1-\sigma+\epsilon-i(t+r))}} \geq 1.$$
Next,
$$\abs{\Gamma_\R(-2ir)\Gamma_\R(1+2\epsilon)}
= (2\pi)^{-\epsilon} \left( \frac{\pi}{r\sinh(\pi r)} \right)^{1/2},$$
and for $\frac{1}{2}\leq\sigma\leq1$,
$$\abs{\frac{(\frac{s+\epsilon+ir}{2}+n)(\frac{1-s+\epsilon+ir}{2}+n)}{
    (ir+1+n)(1+n)}} < 1+\frac{\abs{t^2-r^2}}{4}.$$
Hence for $\cos^2\theta\leq(4+\abs{t^2-r^2})^{-1}$,
$$\F{\frac{s+\epsilon+ir}{2}}{\frac{1-s+\epsilon+ir}{2}}{ir+1}{\cos^2\theta}
= 1 + \sum_{n=1}^\infty \Theta\left(\frac{1}{4^n}\right)
= 1 + \Theta\left( \frac{1/4}{1-1/4} \right),$$
where $\Theta(x)$ stands for a value of absolute size at most $x$.

$r>5$ implies that $\abs{t+r}>5$ or $\abs{t-r}>5$.  Thus
$$\max\abs{g_\theta(s,f),\cc{g_\theta(\cc{s},f)}}
\geq (2\pi)^{-\epsilon} \left( \frac{\pi}{r\sinh(\pi r)} \right)^{1/2}
(1-\Theta(\tfrac{1}{3})).$$
Adjusting the phase factor $(\cos\theta)^{ir}$ in \eqref{eq:gamma_theta}
suitably, i.e.\ taking $\theta=\theta_1$ and $\theta=\theta_2$ with
$(\cos\theta_2)^{2ir}=e^{-i\pi}(\cos\theta_1)^{2ir}$, respectively,
completes the proof.
\end{proof}

\section{Interpolating zeros}\label{s:zeros}

We compute values of $L$ on a grid, but we are ultimately interested in
the zeros, which are not regularly spaced.
To zoom in on the zeros, we interpolate
$$h(t) := e^{-\frac{(t-t_0)^2}{2b^2}} g(t)$$
with $g(t)=\Lambda_\theta(\sigma+it,f)$ and $g(t)=\gamma_\theta(\sigma+it,f)$,
respectively.
The function $h$ has the advantage that it decays rapidly at $\infty$
and is approximately bandlimited,
which allows us to use the Whittaker--Shannon sampling theorem \cite{Brown1967}
$$\abs{h(t) - \sum_{m\in\Z} h\!\left(\frac{m}{A}\right)
  \sinc\!\left(\pi A\left(t-\frac{m}{A}\right)\right)}
\leq 4\int_{\pi A}^\infty \abs{\hat{h}(u)}\,du.$$
Truncating the sum over $m$ and bounding the error of truncation,
we get an effective interpolation formula for $h$.

\begin{lemma}\label{lem:sum Lambda truncated}
For $g(t):=\Lambda_\theta(\sigma+it,f)$,
$G:=\max\limits_{\abs{t}\leq T}\abs{g(t)}$,
$\frac{1}{A}\leq J\leq T-\abs{t_0}$, and $b>0$,
\begin{multline*}
  \abs{\sum_{\abs{\frac{m}{A}-t_0}>J} h\!\left(\frac{m}{A}\right)
    \sinc\left(\pi A\left(t-\frac{m}{A}\right)\right)}
  < \sqrt{2\pi}bA\left\{ G\erfc\left(\frac{J-\frac{1}{A}}{\sqrt{2}b}\right)
  \right. \\ \left.
  + 3N^{1/2}E_{\sigma,\theta,\delta}e^{-(\delta-\theta)T}
  \left\{
  \frac{2b}{\sqrt{2\pi}} e^{-\frac{(T-\abs{t_0}-\frac{1}{A})^2}{2b^2}}
  + (\abs{t_0}+D_{\sigma,f})
  \erfc\left(\frac{T-\abs{t_0}-\frac{1}{A}}{\sqrt{2}b}\right)
  \right\}
  \right\}.
\end{multline*}
\end{lemma}

\begin{proof}
For $\abs{t}\leq T$, we have $\abs{g(t)}\leq G$,
and for $\abs{t}>T$, we bound $\abs{g(t)}$ by
\eqref{Lambda bound}.  Hence,
\begin{align*}
  & \abs{\sum_{\abs{\frac{m}{A}-t_0}>J} h\!\left(\frac{m}{A}\right)
    \sinc\left(\pi A\left(t-\frac{m}{A}\right)\right)}
  \leq \sum_{\abs{\frac{m}{A}-t_0}>J}  e^{-\frac{(\frac{m}{A}-t_0)^2}{2b^2}}
  \abs{g\!\left(\frac{m}{A}\right)}
  \\
  & < \sum_{\substack{\abs{\frac{m}{A}-t_0}>J \\ \abs{\frac{m}{A}}\leq T}}
  e^{-\frac{(\frac{m}{A}-t_0)^2}{2b^2}} G
  + \sum_{\substack{\abs{\frac{m}{A}-t_0}>J \\ \abs{\frac{m}{A}}>T}}
  e^{-\frac{(\frac{m}{A}-t_0)^2}{2b^2}}
	3N^{1/2}\left(\abs{\frac{m}{A}}+D_{\sigma,f}\right)
  E_{\sigma,\theta,\delta}e^{-(\delta-\theta)\abs{\frac{m}{A}}}
  \\
  & < G \int_{\abs{\frac{m}{A}-t_0}>J-\frac{1}{A}}
  e^{-\frac{(\frac{m}{A}-t_0)^2}{2b^2}}\,dm
  + 3N^{1/2} E_{\sigma,\theta,\delta}e^{-(\delta-\theta)T}
  \int_{\abs{\frac{m}{A}}>T-\frac{1}{A}}
  e^{-\frac{(\frac{m}{A}-t_0)^2}{2b^2}}
	\left(\abs{\frac{m}{A}}+D_{\sigma,f}\right)\,dm.
\end{align*}
\end{proof}

\begin{lemma}\label{lem:sum gamma truncated}
For $g(t):=\gamma_\theta(\sigma+it,f)$,
$G:=\max\limits_{\abs{t}\leq T}\abs{g(t)}$,
$\frac{1}{A}\leq J\leq T-\abs{t_0}$, and $b>0$,
\begin{multline*}
  \abs{\sum_{\abs{\frac{m}{A}-t_0}>J} h\!\left(\frac{m}{A}\right)
    \sinc\left(\pi A\left(t-\frac{m}{A}\right)\right)}
  < \sqrt{2\pi}bA\left\{ G\erfc\left(\frac{J-\frac{1}{A}}{\sqrt{2}b}\right)
  \right. \\ \left.
  + E_{\sigma,\theta,\delta}e^{-(\delta-\theta)T}
  \erfc\left(\frac{T-\abs{t_0}-\frac{1}{A}}{\sqrt{2}b}\right) \right\}.
\end{multline*}
\end{lemma}
\begin{proof}
The proof is similar to that of the previous lemma, except that
for $\abs{t}>T$ we bound $\abs{g(t)}$ by Lemma~\ref{lem:gamma}.
\end{proof}

\begin{lemma}\label{lem:int hatLambda}
For $g(t):=\Lambda_\theta(\sigma+it,f)$, $\frac{1}{2}\leq\sigma\leq1$,
$u_0\geq0$, $\pi A\geq u_0$, and $b>0$,
\begin{multline*}
  4\int_{\pi A}^\infty \abs{\hat{h}(u)}\,du
  < 2\abs{c_\theta(\sigma,f)}
  \frac{e^{u_0(\sigma-\frac{1}{2})}}{1-e^{-2\pi\cos(\theta)/\sqrt{N}}}
  \left\{
  \frac{\sqrt{2\pi}}{2b} \erfc\left(\frac{b}{\sqrt{2}}(\pi A-u_0)\right)
  \right. \\ \left.
  + \frac{2e^{-u_0(\sigma-\frac{1}{2})}}{(2\pi\cos\theta)^{\sigma-\frac{1}{2}}}
  \Gamma\left(\sigma-\tfrac{1}{2},2\pi\cos(\theta)e^{u_0}\right).
  \right\}
\end{multline*}
\end{lemma}

\begin{proof}
By Fourier convolution,
$$\hat{h}(v) = \frac{b}{\sqrt{2\pi}} \int_\R
e^{-\frac{b^2}{2}(v-u)^2-i(v-u)t_0} \hat{g}(u)\,du$$
with $\hat{g}(u):=c_\theta(\sigma,f)f(ie^{i\theta}e^u)e^{u(\sigma-\frac{1}{2})}$,
as defined in (\ref{Lambda}a).
Using \eqref{f bound} and writing $a:=2\pi\cos\theta$ gives
$$\abs{\hat{h}(v)} < \frac{b}{\sqrt{2\pi}} \int_\R e^{-\frac{b^2}{2}(v-u)^2}
\frac{\abs{c_\theta(\sigma,f)}e^{u(\sigma-\frac{1}{2})}}{
  e^{a\max\{e^u,e^{-u}/N\}}-1}\,du.$$
Changing the order of integration, we have
\begin{align*}
  4\int_{\pi A}^\infty \abs{\hat{h}(v)}\,dv
  & < \frac{4b}{\sqrt{2\pi}}\abs{c_\theta(\sigma,f)}
  \int_\R \frac{du \, e^{u(\sigma-\frac{1}{2})}}{e^{a\max\{e^u,e^{-u}/N\}}-1}
  \int_{\pi A}^\infty\,dv \, e^{-\frac{b^2}{2}(v-u)^2}
  \\
  & = 2\abs{c_\theta(\sigma,f)} \int_\R \frac{e^{u(\sigma-\frac{1}{2})}
  \erfc\left(\frac{b}{\sqrt{2}}(\pi
	A-u)\right)}{e^{a\max\{e^u,e^{-u}/N\}}-1}\,du.
\end{align*}

Let $u_0\in[0,\pi A]$ and set $x:=\frac{b}{\sqrt{2}}(\pi A-u)$.
For $u<u_0$, we have $x>0$ and $\erfc(x)<e^{-x^2}$, while for
$u\geq u_0$, $\erfc(x)<2$.
Moreover, for $u<u_0$, $a>0$, $\sigma\geq\frac{1}{2}$,
$$\frac{e^{u(\sigma-\frac{1}{2})}}{e^{a\max\{e^u,e^{-u}/N\}}-1}
< \frac{e^{u_0(\sigma-\frac{1}{2})}}{1-e^{-a/\sqrt{N}}},$$
while for $u\geq u_0$,
$$\frac{1}{e^{a\max\{e^u,e^{-u}/N\}}-1}
< \frac{1}{(1-e^{-a/\sqrt{N}})e^{a e^u}}.$$
Thus,
\begin{multline*}
  4\int_{\pi A}^\infty \abs{\hat{h}(v)}\,dv
  < 2\abs{c_\theta(\sigma,f)}
  \frac{e^{u_0(\sigma-\frac{1}{2})}}{1-e^{-a/\sqrt{N}}}
  \left\{ \int_{-\infty}^{u_0} e^{-\frac{b^2}{2}(\pi A-u)^2}\,du
  + 2e^{-u_0(\sigma-\frac{1}{2})} \int_{u_0}^\infty
  \frac{e^{u(\sigma-\frac{1}{2})}}{e^{a e^u}}\,du \right\}.
\end{multline*}
Identifying the integrals with the complementary error function and the
incomplete gamma function completes the proof.
\end{proof}

\begin{lemma}\label{lem:int hatgamma}
For $g(t):=\gamma_\theta(\sigma+it,f)$, $\frac{1}{2}\leq\sigma\leq1$,
$u_0\geq0$, $\pi A\geq u_0$, and $b>0$,
\begin{multline*}
  4\int_{\pi A}^\infty \abs{\hat{h}(u)}\,du
  < 2\abs{c_\theta(\sigma,f)} \left\{
  \frac{2}{b} e^{u_0(\sigma-\frac{1}{2})} (2\pi\cos(\theta)e^{u_0})^{1/6}
  \sech(\frac{\pi r}{2}) \erfc\left(\frac{b}{\sqrt{2}}(\pi A-u_0)\right)
  \right. \\ \left.
  + (2\pi\cos\theta)^{\frac{1}{2}-\sigma}
  \Gamma(\sigma-\tfrac{1}{2},2\pi\cos(\theta)e^{u_0})
  \right\}.
\end{multline*}
\end{lemma}
\begin{proof}
By Fourier convolution,
$$\hat{h}(v) = \frac{b}{\sqrt{2\pi}} \int_\R
e^{-\frac{b^2}{2}(v-u)^2-i(v-u)t_0} \hat{g}(u)\,du,$$
with $\hat{g}(u) := c_\theta(\sigma,f) (\cos(\theta)e^u)^{1/2}
K_{ir}(2\pi\cos(\theta)e^u) \cos^{(-\epsilon)}(2\pi\sin(\theta)e^u)
e^{u(\sigma-1/2)}$, as in (\ref{gamma}a).
For $u<u_0$, we have
$\abs{\cosh(\frac{\pi r}{2})K_{ir}(y)}<4y^{-1/3}$
\cite[p.~107]{BookerStrombergssonThen2013},
$\abs{\cos^{(-\epsilon)}(x)}\leq1$, and
$e^{u(\sigma-\frac{1}{2})}<e^{u_0(\sigma-\frac{1}{2})}$.
Hence, writing $\frac{a}{2\pi}:=\cos\theta>0$, we have
$$\int_{-\infty}^{u_0} e^{-\frac{b^2}{2}(v-u)^2} \abs{\hat{g}(u)}\,du
< \abs{c_\theta(\sigma,f)} e^{u_0(\sigma-\frac{1}{2})}
\int_{-\infty}^{u_0} e^{-\frac{b^2}{2}(v-u)^2}
\frac{4(2\pi)^{-1/3}}{\cosh(\frac{\pi r}{2})}
\left(\frac{a}{2\pi} e^u\right)^{1/6}\,du.$$
For $u\geq u_0\geq 0$, we use
$\abs{K_{ir}(y)}<(\frac{\pi}{2y})^{1/2}e^{-y}$,
and $\abs{\cos^{(-\epsilon)}(x)}\leq1$ to obtain
$$\int_{u_0}^\infty e^{-\frac{b^2}{2}(v-u)^2} \abs{\hat{g}(u)}\,du
< \abs{c_\theta(\sigma,f)} \int_{u_0}^\infty e^{-\frac{b^2}{2}(v-u)^2}
\frac{1}{2} e^{-a e^u} e^{u(\sigma-\frac{1}{2})}\,du.$$
Therefore, after interchanging the order of integration,
\begin{align*}
  4\int_{\pi A}^\infty \abs{h(v)}\,dv < \frac{4b}{\sqrt{2\pi}}
  \abs{c_\theta(\sigma,f)} \left\{
  e^{u_0(\sigma-\frac{1}{2})} \int_{-\infty}^{u_0}
  \frac{4(2\pi)^{-1/2}}{\cosh(\frac{\pi r}{2})} (a e^u)^{1/6}
  \int_{\pi A}^\infty e^{-\frac{b^2}{2}(v-u)^2}\,dv\,du
  \right. \\ \left.
  + \int_{u_0}^\infty \frac{1}{2} e^{-a e^u} e^{u(\sigma-\frac{1}{2})}
  \int_{\pi A}^\infty e^{-\frac{b^2}{2}(v-u)^2}\,dv\,du\right\}
  \\
  < 2\abs{c_\theta(\sigma,f)} \left\{
  e^{u_0(\sigma-\frac{1}{2})} \frac{4(2\pi)^{-1/2}}{\cosh(\frac{\pi r}{2})}
  (a e^{u_0})^{1/6} \int_{-\infty}^{u_0} e^{-\frac{b^2}{2}(\pi A-u)^2}\,du
  + \int_{u_0}^\infty e^{-a e^u} e^{u(\sigma-\frac{1}{2})}\,du
  \right\},
\end{align*}
where we have employed $\erfc(x)\leq e^{-x^2}$ for $x\geq0$, and $\erfc(x)<2$
otherwise.  Evaluating the integrals completes the proof.
\end{proof}

\section{Detecting zeros}\label{s:detect}

For each Maass form $L$-function under consideration, we compute rigorously
many values on the critical line.  For instance, Figure~\ref{fig:Z} shows
a graph of
$$Z(t,f) := L(\tfrac{1}{2}+it,f) e^{i\arg\gamma(\tfrac{1}{2}+it,f)}$$
for the first even Maass form $L$-function on $\Gamma(1)$.

\begin{figure}
\includegraphics{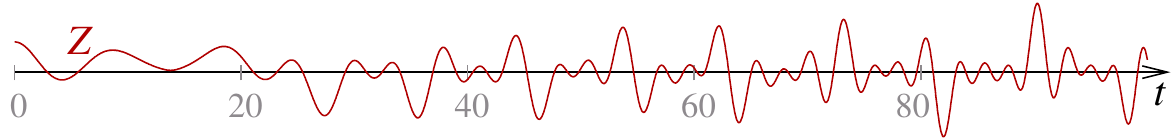}
\caption{\label{fig:Z}Graph of $Z(t,f)$ for the first even Maass form
$L$-function on $\SL(2,\Z)$.}
\end{figure}


\begin{corollary}\label{cor:gamma zeros}
Let $\Omega_{\gamma_\theta} := \{t\in\R:\gamma_\theta(\tfrac{1}{2}+it)=0\}$
be the set of zeros on the critical line of the $\Gamma$-factor.
For values of $\theta_1$ and $\theta_2$ chosen according to Lemma~\ref{lem:min gamma},
$$\Omega_{\gamma_{\theta_1}} \cap \Omega_{\gamma_{\theta_2}} = \emptyset.$$
\end{corollary}
\begin{proof}
This follows immediately from Lemma~\ref{lem:min gamma}.
\end{proof}

\begin{remark}\rm
Fixing the value of $\theta$, say $\theta=\theta_1$, there is a risk
of hitting a zero of $\gamma_{\theta_1}$ (to within the internal
precision) when evaluating $Z_{\theta_1}$ for some specific value of $t$.
In all of our computations, we never observed this in practice,
i.e.\ we never had to deal with division by zero.
However, computing at finite absolute precision, we sometimes come close
to a zero of $\gamma_{\theta_1}$ and experience some loss of precision
in the division by $\abs{\gamma_{\theta_1}}$.
Since we also compute for a second value of $\theta$, $\theta=\theta_2$,
chosen according to Lemma~\ref{lem:min gamma}, we may always ensure the
accuracy of the computed values of $Z$.
\end{remark}

For each Maass form $L$-function under consideration, we rigorously compute all
zeros on the critical line up to some height.
The search for zeros is faciliated by the following lemma.
\begin{lemma}\label{lem:t}
(a) Let $Z\in C^1(\R)$ be real valued, and assume it has consecutive
simple zeros at $t_0$, $t_1$ and $t_2$, with $Z'(t_0)>0$.
Then $\exists t_a,t_b,t_c,t_d,t_e,t_f$ such that the following holds:
$$\begin{array}{c|ccc}
  t                 & Z'(t) & Z(t) & \text{quadrant of $Z'+iZ$} \\
  \hline
  t_a<t<t_0         &  >0   &  <0  &     4                      \\
  t_0<t<t_b         &  >0   &  >0  &     1                      \\
  t_b\leq t\leq t_c &       &  >0  &     1 \text{ or } 2        \\
  t_c<t<t_1         &  <0   &  >0  &     2                      \\
  t_1<t<t_d         &  <0   &  <0  &     3                      \\
  t_d\leq t\leq t_e &       &  <0  &     3 \text{ or } 4        \\
  t_e<t<t_2         &  >0   &  <0  &     4                      \\
  t_2<t<t_f         &  >0   &  >0  &     1                      \\
\end{array}$$

(b) Let $Z\in C^n(\R)$ be real valued, and assume it has a zero of order $n$ at
$t_0$, with
$\left(\tfrac{d^n}{dt^n}Z\right)(t_0)>0$.  Then $\exists t_a,t_b$ such that
the following holds:
$$\begin{array}{c|cccc}
  t & \tfrac{d^n}{dt^n}Z & \tfrac{d^{n-1}}{dt^{n-1}}Z &
  \tfrac{d^{n-2k}}{dt^{n-2k}}Z & \tfrac{d^{n-2k-1}}{dt^{n-2k-1}}Z \\
  \hline
  t_a<t<t_0  &  >0  &  <0  &  >0  &  <0 \\
  t=t_0      &  >0  &  =0  &  =0  &  =0 \\
  t_0<t<t_b  &  >0  &  >0  &  >0  &  >0 \\
\end{array}$$
for $k\in\Z$, but $0<2k<n$.
\end{lemma}
\begin{proof}
The lemma follows from elementary analysis and the intermediate value theorem.
\end{proof}

\begin{algorithm}\label{alg:t}
Let $Z\in C^\infty(\R)$ be real valued.
Let $(t_j)_{j\in\N}$ be a strictly increasing sequence of real numbers.
Refine the sequence $(t_j)$ until all zeros of $Z$ are isolated, i.e.\ there
is at most one zero per interval $(t_j,t_{j+1}]$.
\end{algorithm}

\begin{remark}\label{rem:t}\rm
If the quadrants of $Z'+iZ$ for consecutive $t_j$ are not ordered
as given in Lemma~\ref{lem:t}(a), there is either a zero of order greater
than $1$ which is to be investigated according to Lemma~\ref{lem:t}(b), or the
sequence is not yet fine enough.
We expect the sequence to be fine enough if for successive $t_j$ the
quadrants of $Z'+iZ$ do change by at most by $1$, and when they
change, they do so in agreement with Lemma~\ref{lem:t}.
\end{remark}

There is no proof that the expectation in Remark~\ref{rem:t} holds,
and one can construct sequences $(t_j)$ that contradict the expectation.
Nevertheless, with some reasonable choices in the construction of the
sequence $(t_j)$ and its refinements, the expectation turns out to be
reliable in practice.  Namely, for every Maass form $L$-function that we
considered, we never overlooked any zero, as proven after the fact using
Turing's method.

\section{Turing's method}\label{s:Turing}

Turing's method for verifying the Riemann hypothesis for arbitrary
$L$-functions is described in \cite{Booker2006}.
For $t$ not the ordinate of some zero or pole of $\Lambda$, let
$$S(t) := \frac{1}{\pi} \Im \int_\infty^{\frac{1}{2}} \frac{L'}{L}(\sigma+it,f)
\,d\sigma.$$
By convention, we make $S(t)$ upper semicontinuous, i.e.\ when $t$ is the
ordinate of a zero or pole, we define $S(t)=\lim_{\eps\to0^+} S(t+\eps)$.

We select a particular branch of $\log\gamma(s)$ by using the principal branch
of $\log\Gamma$.  With this choice, set
$$\Phi(t) := \frac{1}{\pi} \Im \log\gamma(\frac{1}{2}+it,f).$$
We further define
$$N(t) := \Phi(t)+S(t),$$
which relates to the number of zeros in the critical strip up to height
$t$.  For $t_1<t_2$ let $\Omega_L$ denote the multiset
of zeros with imaginary part in $(t_1,t_2]$,
and let $N(t_1,t_2)$ denote their number, counting multiplicity,
$$N(t_1,t_2) := \#\Omega_L(t_1,t_2).$$
Then, we have
$$N(t_1,t_2) = N(t_2)-N(t_1).$$

\begin{theorem}\cite[\S4]{Booker2006}\label{thm:Booker4.6}
For $\vartheta=\tfrac{7}{64}$ and $\sigma>\vartheta+1$, define
$$z_\vartheta(\sigma) := \left( \frac{
  \zeta(2\sigma+2\vartheta)\zeta(2\sigma-2\vartheta)}{
  \zeta(\sigma+\vartheta)\zeta(\sigma-\vartheta)} \right)^{1/2}$$
and
$$Z_\vartheta(\sigma)
:= \left( \zeta(\sigma+\vartheta)\zeta(\sigma-\vartheta) \right)^{1/2}.$$
Suppose $t_1$ and $t_2$ satisfy
$$(t_i\pm r)^2\geq\left(\tfrac{5}{2}+\epsilon\right)^2+X^2,\quad i=1,2$$
for some $X>5$, and set
$$C_\vartheta := \log Z_\vartheta\left(\frac{3}{2}\right)
+ \int_{3/2}^\infty \log\frac{Z_\vartheta(\sigma)}{z_\vartheta(\sigma)} d\sigma
- \int_{3/2}^{5/2} \log z_\vartheta(\sigma) d\sigma
+ (\log 4) \frac{z_\vartheta'(\tfrac{3}{2})}{z_\vartheta(\tfrac{3}{2})}.$$
Then
\begin{align*}
  \pi\int_{t_1}^{t_2} S(t)\,dt
  \leq \frac{1}{4}\log\abs{Q\left(\frac{3}{2}+it_2\right)}
  + \left(\log2 - \frac{1}{2}\right)\log\abs{Q\left(\frac{3}{2}+it_1\right)}
  + 2C_\vartheta + \frac{2}{\sqrt{2}(X-5)}.
\end{align*}
\end{theorem}

\begin{corollary}\label{Turing method}
For $0\leq t_1<t_2$, assume $\tilde{\Omega}_L(t_1,t_2)$ is a given
multiset of zeros with imaginary part in $(t_1,t_2]$,
i.e.\ $\tilde{\Omega}_L(t_1,t_2) \subseteq \Omega_L(t_1,t_2)$.
Let
\begin{align*}
  & N_{\tilde{\Omega}_L}(t_1,t_2)
  := \#\tilde{\Omega}_L(t_1,t_2),
  \qquad \text{counting multiplicity},
  \\
  & N_{\tilde{\Omega}_L}(t) := N_{\tilde{\Omega}_L}(t,0)+\Phi(0)+S(0),
  \\
  \text{and} \qquad
  & S_{\tilde{\Omega}_L}(t) := N_{\tilde{\Omega}_L}(t)-\Phi(t).
\end{align*}
If
$$\pi\int_{t_1}^{t_2} \left( S_{\tilde{\Omega}_L}(t) + 1 \right)\,dt$$
exceeds the right-hand side of the bound in Theorem~\ref{thm:Booker4.6},
then the set $\tilde{\Omega}_L(0,t_1)$ contains all zeros with imaginary part
in $(0,t_1]$. $\tilde{\Omega}_L(0,t_1) = \Omega_L(0,t_1)$.
\end{corollary}

\begin{proof}
If $\tilde{\Omega}_L(0,t_1)$ were a proper subset of $\Omega_L(0,t_1)$,
then we would have $N_{\tilde{\Omega}_L}(t_1)<N(t_1)$, whence
$S_{\tilde{\Omega}_L}(t)+1\leq S(t)\ \forall t\in(t_1,t_2]$.
But the integral of the latter is bounded by Theorem~\ref{thm:Booker4.6}.
\end{proof}

\section{Numerical results}\label{s:results}

We consider consecutive Maass cusp forms on $\SL(2,\Z)=\Gamma(1)$.
Booker, Str\"ombergsson, and Venkatesh \cite{BookerStrombergssonVenkatesh2006}
have rigorously computed the first $10$ Maass cusp forms on $\SL(2,\Z)$ to
high precision.  Bian \cite{Bian2017} has extended these computations
to a larger number of Maass cusp forms.
The readily available list of rigorously computed Maass cusp forms is
consecutive for the first $2191$ Maass cusp forms, which covers all
Maass cusp forms whose Laplacian eigenvalue $\lambda=r^2+\frac{1}{4}$
falls into the range $0\leq r\leq178$.

Previous numerical computations of some non-trivial zeros for a few even
Maass form $L$-functions were made by Str\"ombergsson \cite{Strombergsson1999}.
We extend his results by rigorously computing, for each of the first
$2191$ consecutive Maass form $L$-functions on $\SL(2,\Z)$, many values of $Z$,
including all non-trivial zeros up to $T=30000$, at least.

\begin{remark}\rm
At the time of Str\"ombergsson's work, even the numerical data
pertaining to the Maass cusp forms for $\SL(2,\Z)$ was not rigorously proven to be
accurate, so he had no reason to carry out his computations of the zeros
with more than heuristic estimates for the error.
Making use of the rigorous data sets described above, we
have rigorously verified the correctness of Str\"ombergsson's results.
In particular his lists of zeros are consecutive
and accurate.  Moreover, we confirm his observation of a zero-free region on
the critical line for $t$ near $r$, when $r$ is small.

We note that some theoretical results, such as Cho's theorem \cite{Cho2013}
on simple zero of Maass form $L$-functions, assumed the correctness of
Str\"ombergsson's numerical results. With our verification, Cho's
theorem becomes unconditional.
\end{remark}

\begin{table}
\caption{\label{tab:zeros}Consecutive lists of the first few non-trivial zeros
for the first five Maass form $L$-functions on $\SL(2,\Z)$.  Each column is
for one Maass form $L$-function and is specified by the spectral parameter $r$
and the parity $\epsilon$.  The displayed numbers are
the ordinates of the first few consecutive
zeros for $t>0$, all of which are on the critical line.  Each number is
accurate to within $\pm1$ in the last digit.}
\begin{scriptsize}
$$\begin{array}{c}
r=\ 9.533695261 \\
\epsilon= 1 \\
\hline
17.0249420759926 \\
19.3441154991815 \\
22.8261931283343 \\
25.7999235601013 \\
27.6164361749163 \\
29.1018834648622 \\
31.8310613699717 \\
34.3471038793177 \\
35.6095026712633 \\
37.1600794665599 \\
38.9798718247120 \\
40.8649210955904 \\
42.9624682023100 \\
44.7165876388095 \\
45.7081766302651 \\
46.9228865619812 \\
48.8845923479460 \\
50.8585578341632 \\
52.2084561079916 \\
53.8667859217878 \\
54.8124463691756 \\
56.1642766726080 \\
57.7477158040669 \\
59.1886000560169 \\
61.2112906749061 \\
62.4099725413140 \\
63.3997167996275 \\
64.4782740136229 \\
65.8411920277228 \\
67.6680975697523 \\
68.7657311068868 \\
70.5658999093301 \\
71.5151450636631 \\
72.7793106037368 \\
74.0609762360244 \\
74.6049003579295 \\
76.6307909351257 \\
77.8404437657817 \\
79.5177824537089 \\
80.5991976660814 \\
81.2939779667956 \\
82.6785102707012 \\
83.5676015306852 \\
85.3322176044279 \\
86.2362409919154 \\
87.7201460156160 \\
89.2073136143526 \\
90.1509029432393 \\
91.1018169318231 \\
91.9180366781390 \\
93.2577400210036 \\
94.5361681047575 \\
96.0183894125484 \\
97.3452841035982 \\
98.1159568311207 \\
99.4221065922338 \\
100.313745124143 \\
101.182867524182 \\
102.639857458614 \\
103.546439751200 \\
104.953553816157 \\
106.296129428261 \\
107.434291294864 \\
108.433276503701 \\
109.252032855555 \\
109.980276788617 \\
111.243323281265 \\
112.822069704440 \\
113.642166722904 \\
114.945311219697 \\
\end{array}
\begin{array}{c}
r=12.173008325 \\
\epsilon= 1 \\
\hline
5.10553130864728 \\
17.7442287880043 \\
22.0828833772350 \\
23.6900118314732 \\
27.1426126160360 \\
28.8988378613334 \\
31.2199778278305 \\
33.3027699993553 \\
34.9413315016281 \\
36.8220610290123 \\
39.4036457550467 \\
40.3954506308287 \\
41.7518913966523 \\
44.3141127846671 \\
45.6041247768810 \\
46.7731555096729 \\
49.0623573859669 \\
50.4428786981306 \\
51.4839637477090 \\
53.1358251106050 \\
54.8229659148021 \\
56.5390086739774 \\
57.7030588658215 \\
58.8511658803142 \\
60.3363236494317 \\
62.3427605114013 \\
63.1451180474117 \\
64.2165899608009 \\
66.2554993505022 \\
67.7507046815066 \\
68.5590956746215 \\
69.4453333319929 \\
71.3077160409394 \\
72.9519709018265 \\
73.6890121026987 \\
75.2952196918033 \\
76.2065113468558 \\
77.7272898984985 \\
78.9591932615256 \\
80.2505964617052 \\
81.2013544790962 \\
82.3915692402408 \\
84.2017334394370 \\
85.5004972002107 \\
86.0646875741528 \\
86.9821021413084 \\
88.6357429075844 \\
90.2062740406416 \\
91.1589654673596 \\
91.8841682334637 \\
93.3261418331635 \\
94.8381113365867 \\
95.6998921488729 \\
96.8603035570679 \\
97.8041648452534 \\
99.0948448347175 \\
100.185735184243 \\
101.737682026863 \\
102.777547784678 \\
103.612872790404 \\
104.586408006086 \\
105.618614451690 \\
107.487452706483 \\
108.204341251653 \\
109.099573594018 \\
110.023692551348 \\
111.610332958436 \\
112.968244330445 \\
113.535660411511 \\
114.803248583420 \\
\end{array}
\begin{array}{c}
r=13.779751352 \\
\epsilon= 0 \\
\hline
2.89772467827094 \\
5.59124531531950 \\
21.0903775087339 \\
23.1575104845853 \\
25.4393003898372 \\
29.1892067135368 \\
31.0617394845440 \\
32.4527182375570 \\
34.0272796838472 \\
36.9312371974937 \\
38.9870982151186 \\
40.4655490222834 \\
41.6851103312465 \\
43.0510814899645 \\
45.2203620069413 \\
47.6607243153047 \\
48.8179663907847 \\
49.7984652829980 \\
51.3751449154626 \\
52.5598876963433 \\
54.6535140546208 \\
56.6899697503172 \\
57.6166211934090 \\
59.0433361195422 \\
60.2512945420134 \\
60.9302544966805 \\
63.0554036340306 \\
65.0280616017899 \\
66.0531445397070 \\
67.5747312567319 \\
68.2882811318913 \\
69.4223658112824 \\
70.8802548329208 \\
72.2899933699866 \\
74.1574305286740 \\
75.5105873385132 \\
76.5116524370659 \\
77.2965875092772 \\
78.5268435032172 \\
79.6462440380120 \\
81.0934750546557 \\
83.0999731356220 \\
83.9810907179412 \\
85.1204644360696 \\
86.2039221458232 \\
87.2818523030882 \\
88.1416316506303 \\
89.2944208770600 \\
91.2299257008873 \\
92.3569273378821 \\
93.8884516137452 \\
94.4624382674195 \\
95.5004244761379 \\
96.5910720591420 \\
97.3209962526615 \\
99.1562706753946 \\
100.163498257903 \\
101.758162340822 \\
103.080823014291 \\
103.400716986936 \\
104.807000384562 \\
105.510967974394 \\
106.653239218716 \\
107.677935083655 \\
109.354544636483 \\
110.858892769331 \\
111.635310785971 \\
112.452947413326 \\
113.536311309751 \\
114.599448047724 \\
\end{array}
\begin{array}{c}
r=14.358509518 \\
\epsilon= 1 \\
\hline
3.76470190452593 \\
8.44187034414965 \\
19.4483544500909 \\
23.0939623051538 \\
26.1518661201714 \\
27.8260578322407 \\
30.9903075480748 \\
32.0681458350820 \\
35.0463081449147 \\
36.3730961243226 \\
38.1758857326494 \\
40.0803022802362 \\
42.1472089149297 \\
43.4072270007604 \\
45.2854547559751 \\
46.7361221975669 \\
48.4727303964927 \\
49.9790279360135 \\
51.2742978087781 \\
53.5712937366133 \\
54.6096034975762 \\
55.2759761303777 \\
57.8055604312075 \\
58.9190985459828 \\
60.0205118250084 \\
61.9578406743725 \\
62.9609412152269 \\
64.4328958857527 \\
65.8735219313308 \\
66.9654292253422 \\
68.3575794456240 \\
70.0897529037803 \\
71.2206506835744 \\
72.2540257118982 \\
73.6822062524674 \\
74.7169239115014 \\
76.7494469809889 \\
77.4226197121217 \\
78.3595482351092 \\
80.2707645343853 \\
81.6367380140678 \\
82.2630761619909 \\
83.3558394239428 \\
85.1612109448322 \\
86.1020003845199 \\
87.2123216652190 \\
88.8839330774639 \\
89.7818929416073 \\
90.5922255118720 \\
91.9914808989279 \\
93.8202261494019 \\
94.0923380413559 \\
95.3961549935688 \\
96.8178940746616 \\
97.9968741052357 \\
99.2448562396670 \\
100.278579859458 \\
100.957059993630 \\
102.300159307988 \\
103.804155362774 \\
104.874662606167 \\
105.730976498822 \\
106.722239502391 \\
108.413917430577 \\
109.104485482787 \\
110.011383725637 \\
111.404230185118 \\
112.677396867124 \\
113.308306899293 \\
114.342832613189 \\
\end{array}
\begin{array}{c}
r=16.138073172 \\
\epsilon= 1 \\
\hline
4.07043016260804 \\
6.01471804932679 \\
11.9094970989896 \\
22.3497093300588 \\
23.6756749096999 \\
27.7899319219294 \\
30.0381347329908 \\
32.0229736589354 \\
33.8112506234903 \\
35.5014869710102 \\
38.7878732480148 \\
39.9476421272383 \\
41.0312681795598 \\
42.7546214934681 \\
45.1517654717050 \\
46.4771782207275 \\
48.3625255758408 \\
50.1639512295890 \\
51.0355332480426 \\
51.9925909205271 \\
54.2915625621262 \\
56.2796983417801 \\
57.7013302737383 \\
58.3383624264988 \\
59.6802583964633 \\
61.5356150830096 \\
62.7868331792340 \\
63.9997915377463 \\
66.1489856985468 \\
67.3268819458017 \\
68.3016826900454 \\
69.2840066042359 \\
70.2797560418629 \\
72.4309724746729 \\
74.1819479588922 \\
74.7011406259552 \\
76.1335791083958 \\
77.4636368333900 \\
78.4634979258068 \\
79.6781395629882 \\
80.9175966370216 \\
82.2205366713458 \\
84.1953372371218 \\
85.3999580853391 \\
85.8048466006321 \\
86.8146478765005 \\
87.8727881608657 \\
89.8827782231190 \\
90.7761256889692 \\
92.0942449809388 \\
93.3990695897775 \\
94.2940427772870 \\
95.7412459465434 \\
96.5775317444622 \\
97.4404107576090 \\
98.4151717785236 \\
100.248850160768 \\
101.770705877070 \\
102.681563321248 \\
103.496925840063 \\
104.349526888882 \\
105.492038316423 \\
106.554052363134 \\
108.328860724077 \\
108.889566928655 \\
109.993076784022 \\
111.644210582933 \\
112.766801484652 \\
113.685480585879 \\
114.136514740415 \\
\end{array}$$
\end{scriptsize}
\end{table}

Our lists of zeros contain more than $60000$ consecutive non-trivial zeros
per Maass form $L$-function.  All these zeros are simple.
The first several zeros of the first five Maass form $L$-functions
are listed in Table~\ref{tab:zeros}.

\begin{theorem}\label{thm:finite GRH}
For $f$ a Maass cusp form on $\SL(2,\Z)$ with spectral parameter
$0\leq r\leq178$, all non-trivial zeros with $0\leq t\leq30000$ of
the corresponding Maass form $L$-function are simple and on the critical line.
\end{theorem}
\begin{proof}
For each Maass form $L$-function we prove, using Corollary~\ref{Turing method},
that the corresponding list of rigorously computed zeros is consecutive for
$0\leq t\leq30000$, and that all the zeros are indeed simple and
on the critical line.
\end{proof}

According to a conjecture of Montgomery \cite{Montgomery1973}, the
distribution of non-trivial zeros should follow random matrix theory (RMT)
predictions.  In case of Maass form
$L$-functions, the distribution of non-trivial zeros is expected to
conform to that of eigenvalues of large random matrices from the
Gaussian unitary ensemble (GUE) \cite{KeatingSnaith2000}.
This raises the question of how GUE statistics relate to the
zero-free region around $t=r$ observed by Str\"ombergsson
\cite{Strombergsson1999}---are the GUE statistics asymptotically correct
in the large $t$ aspect only?

We investigate this question by distinguishing between zeros with small
and large absolute ordinate,
respectively.  For a given Maass form $L$-function there are only a finite
number of zeros with small ordinate,
and the resulting statistics would be poor.
Knowing the zeros for many Maass form $L$-functions, we can evaluate on a
common scale the distribution of zeros for each $L$-function and
collate the statistics of many of them together.

Let $f$ be a Maass cusp form with spectral parameter $r$ and parity
$\epsilon$.  Consider the zeros of the associated Maass form $L$-function.
We unfold the zeros,
$$x_i := \Phi(t_i),$$
in order to obtain rescaled zeros $x_i$ with a unit mean density.  Then
$s_i := x_{i+i}-x_i$
defines the sequence of nearest-neighbor spacings, which has mean
value $1$ as $i\to\infty$.
Now, the distribution of nearest-neighbor spacings is given by
$$\int_0^s P_f(s')\,ds' := \lim_{j\to\infty} \frac{\#\{i\leq j:s_i\leq s\}}{
  \#\{i\leq j\}},$$
where the index $f$ denotes the corresponding Maass cusp form.
Distributions of rescaled nearest-neighbor spacings
are expected to be independent of the specific parameter values of
corresponding Maass cusp forms $f$ and can be collated by writing
$$P(s):=\frac{1}{\#\{f\}} \sum_{f} P_f(s).$$

\begin{figure}
\includegraphics{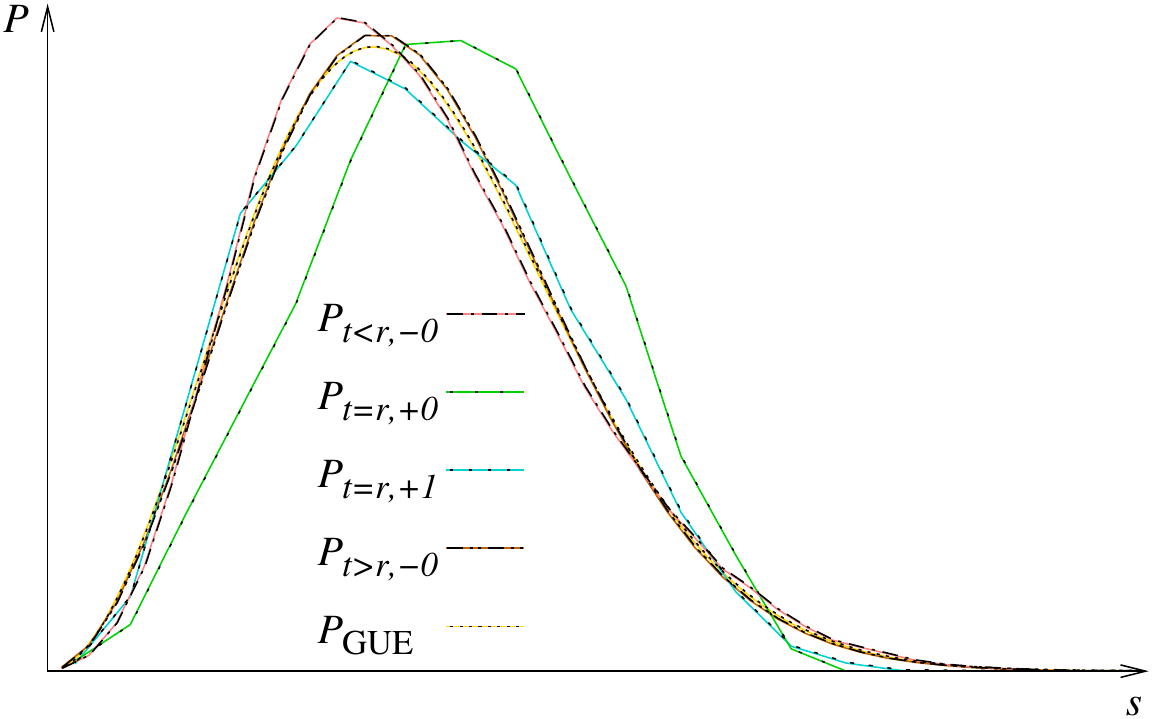}
\caption{\label{fig:P}Rescaled nearest-neighbor spacing distributions
  $P_{t<r,-0}$, $P_{t=r,+0}$, $P_{t=r,+1}$, and $P_{t>r,-0}$, respectively,
  in comparison with the Wigner surmise $P_{\text{GUE}}$.
  Only the distribution of zeros that are in absolute size closest to the
  value of the spectral parameter might show a stronger level repulsion
  than the Wigner surmise.  In all other cases, the distribution of zeros
  closely resembles GUE statistics.}
\end{figure}

To distinguish between zeros with small and large absolute ordinate,
we define the respective nearest-neighbor spacings distributions,
\begin{align*}
  & \int_0^s P_{f,t<r,-n}(s')\,ds'
  := \lim_{j\to\infty} \frac{\#\{i\leq j:s_i\leq s,
    \ 0<t_{i-1},\ t_{i+n-1}<r\}}{\#\{i\leq j:0<t_{i-1},\ t_{i+n-1}<r\}},
  \\
  & \int_0^s P_{f,t=r,+n}(s')\,ds'
  := \lim_{j\to\infty} \frac{\#\{i\leq j:s_i\leq s,\ t_{i-1-n}<r<t_{i+n}\}}{
    \#\{i\leq j:t_{i-1-n}<r<t_{i+n}\}},
  \\
  & \int_0^s P_{f,t>r,-n}(s')\,ds'
  := \lim_{j\to\infty} \frac{\#\{i\leq j:s_i\leq s,\ t_{i-n}>r\}}{
    \#\{i\leq j:t_{i-n}>r\}},
\end{align*}
where $n$ is a non-negative integer, as well as their collated versions
\begin{align*}
  & P_{t<r,-n}(s) := \frac{1}{\#\{f\}} \sum_{f} P_{f,t<r,-n}(s),
  \\
  & P_{t=r,+n}(s) := \frac{1}{\#\{f\}} \sum_{f} P_{f,t=r,+n}(s),
  \\
  & P_{t>r,-n}(s) := \frac{1}{\#\{f\}} \sum_{f} P_{f,t>r,-n}(s).
\end{align*}
For the first $2191$ Maass form $L$-functions on $\SL(2,\Z)$, the resulting
distributions are displayed in Figure~\ref{fig:P}, in comparison with
the Wigner surmise
$$P_{\text{GUE}}(s) := \frac{32}{\pi^2} s^2 e^{-\frac{4s^2}{\pi}}.$$
As is visible, the distribution of zeros resembles GUE statistics
for both small and large absolute ordinate, and
there appears to be no distinction between the statistics of the two
cases.
Only the distribution of zeros that are in absolute size closest to the
value of the spectral parameter might show a stronger level repulsion
than the Wigner surmize.

However, it is unclear whether this seemingly stronger level
repulsion is just an artefact of the limited number of $2191$ spacings that
contribute to the histogram of $P_{t=r,+0}$. If we take three times as many
spacings into account, as is the case with $P_{t=r,+1}$, we again find a
close resemblance to the Poisson distribution.
We speculate that the GUE statistics hold for all $t$, not only
in the large $t$ aspect.

\begin{figure}
\includegraphics{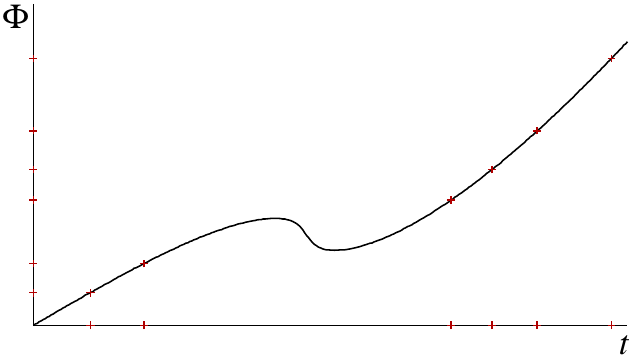}
\caption{\label{fig:Phi}Average number $\Phi$ of non-trivial zeros for the
  first even Maass form $L$-function on $\SL(2,Z)$.  For comparison, the
  locations of zeros and unfolded zeros are also included
  as tics.  Clearly visible is the negative density region
  ($\Phi'<0$) for $t$ around $r$.  Zeros are pulled away from this region
  resulting in a zero-free region in the $t$ aspect,
  but not in the unfolded zeros.}
\end{figure}

Since the GUE statistics are based on rescaled zeros, $x_i=\Phi(t_i)$,
they do not contradict a zero-free region on the critical
line.  The density of zeros is described by $\Phi'$, and according to the
$\Gamma$-factor, the density of zeros is smaller for $t$ in a neighborhood
of $r$.  In particular, for small values of the spectral parameter $r$,
the density $\Phi'$ becomes negative for $t$ near $r$; see
Figure~\ref{fig:Phi}.  There are finitely many Maass form $L$-functions
on $\SL(2,\Z)$ that have such a region where $\Phi'$ is negative.
By inspection, we find that no zero falls
into a negative density region.  Moreover, the zeros seem to be repelled away
from the negative density regions resulting in the zero-free region
around $r$.

\begin{figure}
\includegraphics{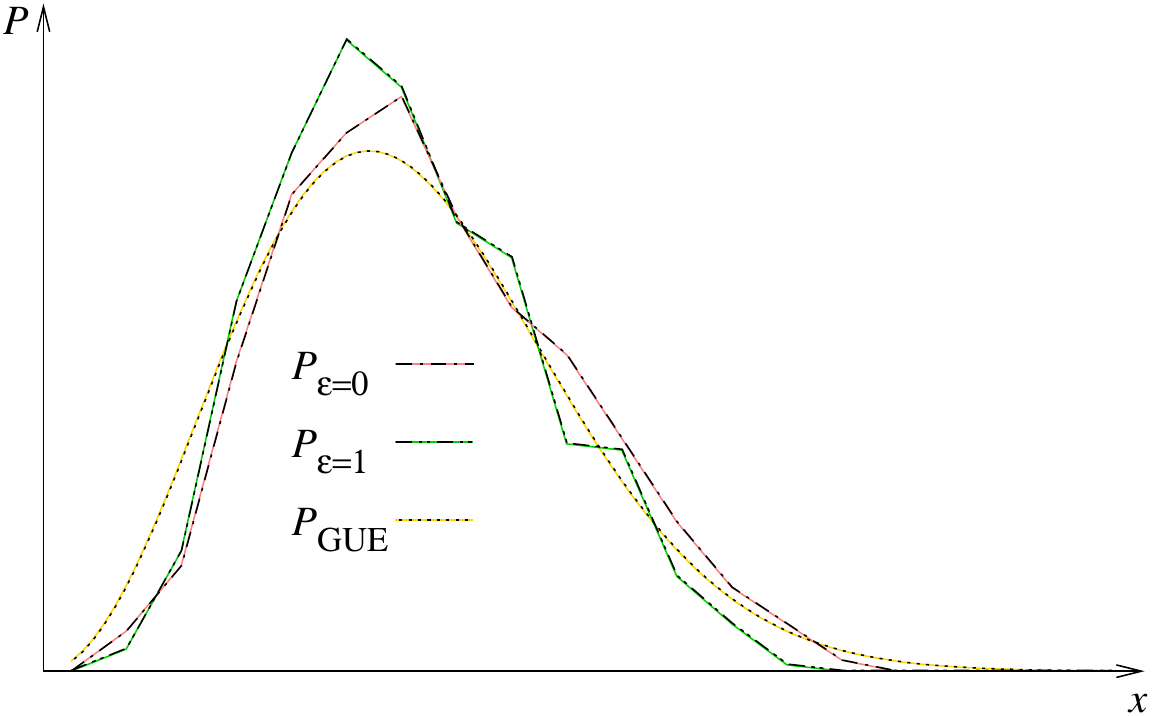}
\caption{\label{fig:P1}Distributions of the rescaled first zero in dependence
  of the parity of the Maass form $L$-function.
  Amongst the $2191$ Maass cusp forms under consideration,
  $1018$ of them are even with respect to reflection in the imaginary axis,
  $\epsilon=0$, and $1173$ of them are odd, $\epsilon=1$.
  In comparison with the Wigner surmise $P_{\text{GUE}}$, close to the origin
  of the plots, the first zero shows a stronger level repulsion
  irrespective of the parity.}
\end{figure}

Finally, we investigate the repulsion from zero of the rescaled first zero
$x_1$ in dependence of the parity $\epsilon$ of the Maass form $L$-function.
For this we consider the distributions of the rescaled first zero,
\begin{align*}
  \int_0^x P_{\epsilon=e}(x')\,dx'
  := \lim_{r\to\infty} \frac{\#\{j: x_1<x,\ r_j\leq r,
    \ \epsilon=e\}}{\#\{j:r_j\leq r,\ \epsilon=e\}},
\end{align*}
for $e\in\{0,1\}$.
The resulting distributions are displayed in Figure~\ref{fig:P1}.
Close to the origin of the plots, they show a stronger level repulsion than the
Wigner surmise $P_{\text{GUE}}(x)$.

\end{document}